\documentclass[reqno]{amsart}

\usepackage{enumitem}
\usepackage{amssymb,amsmath,mathtools}
\usepackage{microtype}
\usepackage{bbm}
\usepackage{hyperref}

\newcommand*{\mailto}[1]{\href{mailto:#1}{\nolinkurl{#1}}}
\newcommand{\arxiv}[1]{\href{http://arxiv.org/abs/#1}{arXiv:#1}}

\newcommand{\msc}[1]{\href{http://www.ams.org/msc/msc2010.html?t=&s=#1}{#1}}

\newtheorem{theorem}{Theorem}[section]
\newtheorem{lemma}[theorem]{Lemma}
\newtheorem{corollary}[theorem]{Corollary}
\newtheorem{remark}[theorem]{Remark}
\theoremstyle{definition}

\newcommand{\R}{{\mathbb R}}

\newcommand{\Z}{{\mathbb Z}}
\newcommand{\C}{{\mathbb C}}

\newcommand{\cL}{{\mathcal L}}
\newcommand{\cF}{{\mathcal F}}

\newcommand{\id}{{\mathbbm{1}}}
\newcommand{\OO}{{\mathcal O}}
\newcommand{\cW}{{\mathcal W}}

\newcommand{\dom}{\mathrm{dom}}
\newcommand{\gt}{\mathfrak{t}}

\newcommand{\I}{\mathrm{i}}
\newcommand{\E}{\mathrm{e}}

\newcommand{\be}{\begin{equation}}
\newcommand{\ee}{\end{equation}}

\newcommand{\wt}{\widetilde}

\newcommand{\lam}{\lambda}

\newcommand\cS{{\mathcal{S}}}

\newcommand{\hyp}[5]{\,\mbox{}_{#1}F_{#2}\!\left(
  \genfrac{}{}{0pt}{}{#3}{#4};#5\right)}

\makeatletter
\newcommand{\dlmf}[1]{%
\cite[%
 \def\nextitem{\def\nextitem{, }}%
 \@for \el:=#1\do{\nextitem\expandafter\dlmf@eq@href\el...\end}%
]{dlmf}%
}
\def\dlmf@eq@href#1.#2.#3.#4\end{%
  \href{http://dlmf.nist.gov/#1.#2.E#3}{(#1.#2.#3)}}
\makeatother

\numberwithin{equation}{section}

\begin{document}

\title[The Discrete Laguerre Operator]{Heat Kernels of the Discrete Laguerre Operators}

\author[A. Kostenko]{Aleksey Kostenko}
\address{Faculty of Mathematics and Physics\\ University of Ljubljana\\ Jadranska ul.\ 19\\ 1000 Ljubljana\\ Slovenia\\ and 
Institute for Analysis and Scientific Computing\\ Vienna University of Technology\\ Wiedner Hauptstra\ss e 8-10/101\\1040 Vienna\\ Austria}
\email{\mailto{Aleksey.Kostenko@fmf.uni-lj.si}}
\urladdr{\url{https://www.fmf.uni-lj.si/~kostenko/}}

\thanks{{\it Research supported by the Slovenian Research Agency (ARRS) under Grant No.\ N1-0137
and the Austrian Science Fund (FWF) under Grant No.\ P28807}}

\thanks{Lett.\ Math.\ Phys., {\it to appear}}


\begin{abstract}
For the discrete Laguerre operators we compute explicitly the corresponding heat kernels by expressing them with the help of Jacobi polynomials. This enables us to show that the heat semigroup is ultracontractive and to compute the corresponding norms. On the one hand, this helps us to answer basic questions (recurrence, stochastic completeness) regarding the associated Markovian semigroup. On the other hand, we prove the analogs of the Cwiekel--Lieb--Rosenblum and the Bargmann estimates for perturbations of the Laguerre operators, as well as  the optimal Hardy inequality.
\end{abstract}

\maketitle

\section{Introduction}\label{sec:Intro}

Our main objects of study are the  discrete Laguerre operators
\be\label{eq:H0}
H_\alpha := \begin{pmatrix} 
1+\alpha & -\sqrt{1+\alpha} & 0 & \cdots \\[1mm]
-\sqrt{1+\alpha} & 3 + \alpha & -\sqrt{2(2+\alpha)}  & \ddots \\[1mm]
0 & -\sqrt{2(2+\alpha)} & 5 + \alpha  & \ddots \\[1mm]
0 & 0 & -\sqrt{3(3+\alpha)} &  \ddots \\
\vdots&\ddots&\ddots&\ddots
\end{pmatrix},\quad \alpha>-1,
\ee
acting in $\ell^2(\Z_{\ge0})$. Explicitly, $H_\alpha = \big(h^{(\alpha)}_{n,m}\big)_{n,m\ge0}$ with $h^{(\alpha)}_{n,m}=0$ if $|n-m|>1$ and 
\begin{align*}
h^{(\alpha)}_{n,n} = 2n+1 + \alpha,\quad  h^{(\alpha)}_{n,n+1} = h_{n+1,n}^{(\alpha)}= - \sqrt{(n+1)(n+1+\alpha)},\quad n\in \Z_{\ge0}.
\end{align*}
It is a special case of a self-adjoint Jacobi operator whose generalized eigenfunctions are precisely the Laguerre polynomials $L_n^{(\alpha)}$,
explaining the name for \eqref{eq:H0}.

The operator $H_\alpha$ features prominently in the study of nonlinear waves in $(2+1)$-dimensional noncommutative scalar field theory \cite{a06, a13, gms}. The coefficient $\alpha$ in \eqref{eq:H0} can be seen as a measure of the delocalization of the field configuration and it is related to the planar angular momentum \cite{a13}. In particular, $\alpha=0$ corresponds to spherically symmetric waves and it has attracted further interest in \cite{cfw03,ks15a,ks15b,ks15c}, where $H_0$ appears as the linear part in the nonlinear Schr\"odinger equation \cite{ks15a,ks15b,ks15c}.
Thus dispersive estimates for the unitary evolution play a crucial role in the understanding of stability of soliton manifolds appearing in these models. It turned out (see \cite{kt16,kkt18}) that the unitary evolution $\E^{\I t H_\alpha}$ can be expressed by means of Jacobi polynomials (see Appendix \ref{app:jacobi} for definitions and basic facts) and this also connects dispersive estimates with uniform weighted estimates of Jacobi polynomials on the orthogonality interval (the so-called Bernstein-type inequalities). 

In the present article we focus on the study of the heat semigroup $(\E^{-t H_\alpha})_{t>0}$. Usually (sharp) dispersive estimates (for $\E^{\I t H_\alpha}$ these are obtained in \cite{kkt18,kt16}) do not imply (sharp) heat kernel estimates as the example of the free Hamiltonian shows. Namely, let $J_0$ be defined in $\ell^2(\Z)$ by 
\be 
(J_0u)_n := -u_{n-1} +2u_n - u_{n+1},\quad n\in \Z.
\ee
$J_0$ is a bounded self-adjoint operator, whose spectrum is purely absolutely continuous and coincides with the interval $[0,4]$. The corresponding heat semigroup and the unitary evolution are given by
\begin{align}
\E^{-tJ_0}(n,m) & = \E^{-2t}I_{n-m}(2t), & 
\E^{\I tJ_0}(n,m) & = \E^{2\I t} I_{n-m}(2\I t),
\end{align}
for all $n,m\in\Z$. Here 
\be
I_k(z) = \I^{-k} J_k(\I z) =  \sum_{n=0}^\infty \frac{1}{n! \Gamma(n+k+1)}\left(\frac{z}{2}\right)^{2n+k}
\ee
 is the modified Bessel function of the first kind \dlmf{10.25.2} (we use the convention $1/\Gamma(m) =0$ if $m\in\Z_{\le 0}$). This leads to the following bounds
\be
\|\E^{\I tJ_0}\|_{\ell_1\to\ell_\infty} = \sup_{n,m\in\Z} |\E^{\I tJ_0}(n,m)| = \OO(|t|^{-1/3})
\ee
as $t\to\infty$, however, 
\be
\|\E^{-tJ_0}\|_{\ell_1\to\ell_\infty} = \OO(t^{-1/2}),\quad t\to+\infty.
\ee

It is not at all surprising that the heat kernel of $\E^{-t H_\alpha}$ is expressed by means of Jacobi polynomials (Theorem \ref{thm:explicit}). However, now one is led to the study of Jacobi polynomials outside of the orthogonality interval. 
Let us next briefly outline the structure of the paper and the main results.

Section \ref{sec:II} is of preliminary character, where we recall the definition of $H_\alpha$ and its basic spectral properties. 

In Section \ref{sec:III} we investigate the quadratic form $\gt_\alpha$ associated with $H_\alpha$. 
Using a convenient factorization of the matrix \eqref{eq:H0} (which connects $H_\alpha$ with the spectral theory of Krein strings, see Remark \ref{rem:string}), we are able to perform a rather detailed study of $\gt_\alpha$ (Lemma \ref{lem:form}). Using the Beurling--Deny criteria, this helps us to conclude that the heat semigroup $\E^{-tH_\alpha}$ is positivity preserving. Moreover, it is Markovian if $\alpha=0$ (that is, $\E^{-t H_0}$ is also $\ell^\infty$ contractive). The string factorization also shows that a very simple similarity transformation \eqref{eq:wtHalpha} connects $H_\alpha$ with the difference operator $\wt{H}_\alpha$, which is Markovian, however, acts in a weighted $\ell^2$ space.  

We investigate heat semigroups $\E^{-tH_\alpha}$ and $\E^{-t\wt{H}_\alpha}$ in Section \ref{sec:IV}. First, we compute explicitly the corresponding heat kernels (Theorem \ref{thm:explicit}).  On the one hand, the connection with Jacobi polynomials enables us to obtain the on-diagonal estimates for the heat kernels (Theorem \ref{thm:HeatEst}). On the other hand, this allows us to show that the continuous time random walk on $\Z_{\ge 0}$ generated by $\wt{H}_\alpha$ is recurrent exactly when $\alpha>0$. Moreover, it is stochastically complete for all $\alpha>-1$. It is interesting to mention that the latter is a consequence of the formula for the generating function of Meixner polynomials (see Remark \ref{rem:heat=Meixner} and Lemma \ref{lem:RW}). Let us stress in this connection that orthogonality relations for Meixner polynomials are equivalent to the unitarity of $\E^{-\I tH_\alpha}$ (see \cite[Remark 3.2]{kkt18}). 

In the final Section \ref{sec:V} we study the negative spectrum of perturbations $H_{\alpha,V}$ of $H_\alpha$. Rank one perturbations of $H_\alpha$ enjoy a very detailed treatment (Lemma \ref{lem:rank1est}). This has several consequences. First of all, for $\alpha\in (-1,0]$ this immediately implies that no matter how small the attractive perturbation $V$ is, it always produces a non-empty negative spectrum (i. e., the presence of a zero energy resonance for $\alpha\in (-1,0]$, cf. \eqref{eq:Greenat0}, which can also be seen as another instance of recurrence). For $\alpha >0$, we can show that for sufficiently small attractive perturbations $V$, the negative spectrum of $H_{\alpha,V}$ remains empty. The qualitative measure of ``smallness" is demonstrated by the optimal Hardy inequality (Theorem \ref{th:Hardy}) as well as by two estimates \eqref{eq:CLR} and \eqref{eq:bargm}. The latter is the analog of the Bargmann bound for 1D Schr\"odinger operators. The former is the analog of the Cwikel--Lieb--Rosenblum bound and it actually follows from the ultracontractivity estimate \eqref{eq:HeatEst1} (indeed, by theorem of Varopoulos, \eqref{eq:HeatEst1} is equivalent to the Sobolev-type inequality \eqref{eq:Sobalpha}, which is known to be further equivalent to a CLR-type bound, \cite{ls97,fls}). Let us stress that the optimal constant $C(\alpha)$ in \eqref{eq:CLR} remains an open problem.  In conclusion let us mention that all the above results resemble a strong similarity between discrete Laguerre operators and 1D radial Schr\"odinger operators (for instance, one may interpret \eqref{eq:CLR} as a discrete analog of the Glaser--Gr\"osse--Martin--Thirring bound \cite[Theorem XIII.9(c)]{rsIV}).  

\subsection*{Notation}
$\R$ and $\C$ have the usual meaning; $\R_{>0} := (0,\infty)$,  $\R_{\ge 0} := [0,\infty)$, and $\Z_{\ge a} := \Z\cap [a,\infty)$ for any $a\in\R$.

By $\Gamma$ is denoted the classical gamma function \dlmf{5.2.1}.
For $x\in\C$ and $n\in\Z_{\ge0}$
\be\label{K6}
(x)_n:=x(x+1)\cdots(x+n-1)\quad(n>0),\quad (x)_0:=1;\quad
\binom{n+x}{n} := \frac{(x+1)_n}{n!}
\ee
denote the {\em Pochhammer symbol} \dlmf{5.2.4}
and the {\em binomial coefficient}, respectively. Notice that for $-x\notin \Z_{\ge0}$
\begin{align*}
(x)_n = \frac{\Gamma(x+n)}{\Gamma(x)},\qquad 
\binom{n+x}{n} = \frac{\Gamma(x+n+1)}{\Gamma(x+1)\Gamma(n+1)}\,.
\end{align*}
Moreover, the above formulas allow to define the Pochhammer symbol and the binomial coefficient for noninteger $x$, $n>0$. 
Finally, for $-c \notin \Z_{\ge0}$
the
{\em Gauss hypergeometric function} \dlmf{15.2.1} is defined by 
\be\label{K7}
\hyp21{a,b}{c}{z} := \sum_{k=0}^\infty \frac{(a)_k(b)_k}{(c)_k k!}\,z^k
\quad\mbox{($|z|<1$ or else $-a$ or $-b\in \Z_{\ge0}$).}
\ee

For a sequence of positive reals $\sigma =( \sigma_n)_{n\ge 0} \subset \R_{>0}$ and $p\in [1,\infty)$, we denote by $\ell^p(\sigma) = \ell^p(\Z_{\ge 0};\sigma)$ the usual weighted Banach space of sequences $u = (u_n)_{n\ge 0}\subset \C$ such that 
\begin{align*}
\|u\|_{\ell^p(\sigma)} = \Big( \sum_{n\ge 0} |u_n|^p \sigma_n \Big)^{1/p} < \infty.
\end{align*}
If $p=\infty$, then the corresponding norm is given by
\begin{align*}
\|u\|_{\ell^\infty(\sigma)} =\sup_{n\ge 0} |u_n| \sigma_n.
\end{align*}
We shall simply write $\ell^p = \ell^p(\Z_{\ge 0})$ if $\sigma = \id$. 
Finally, $\delta_n = (\delta_{n,k})_{k\ge 0}$,  $n\in \Z_{\ge 0}$, is the standard orthonormal basis in $\ell^2(\Z_{\ge 0})$, where $\delta_{n,k}$ is Kronecker's delta.

\section{The discrete Laguerre operator}\label{sec:II} 

We start with a precise definition of the operator $H_\alpha$. 
 For a sequence $u=(u_n)_{n\ge 0}$ we define the difference expression $\tau_\alpha\colon u\mapsto \tau_\alpha u$ by setting
\be\label{eq:tau}
(\tau_\alpha u)_n : = - \sqrt{n(n+\alpha)}\,u_{n-1} + (2n+1 +\alpha)u_n -
\sqrt{(n+1)(n+1+\alpha)}\,u_{n+1},
\ee
for all $n\in\Z_{\ge 0}$, where $u_{-1}:=0$ for notational simplicity. 
Then the operator $H_\alpha$ associated with the Jacobi matrix \eqref{eq:H0} is defined by
\begin{align}\begin{split}
H_\alpha\colon \begin{array}[t]{lcl} \mathcal{D}_{\max} &\to& \ell^2(\Z_{\ge0}) \\ u &\mapsto& \tau_\alpha u\ , \end{array}
\end{split}\end{align}
where 
\be
 \mathcal{D}_{\max} = \{u\in \ell^2(\Z_{\ge0})\ |\, \tau_\alpha u\in \ell^2(\Z_{\ge0})\}.
\ee
Notice that $\mathcal{D}_{\max}$ does not depend on $\alpha$, however, seems, a closed description of $\mathcal{D}_{\max}$ is a rather complicated task.

Spectral properties of $H_\alpha$ are well known. Let us briefly describe them. First of all, the Carleman test (see, e.g., \cite[p.~24]{akh}) implies that $H_\alpha$ is self-adjoint. Moreover, the polynomials of the first kind for \eqref{eq:tau} are given by (see \cite[(5.1.10)]{sz})
\be\label{eq:P=L}
P_{\alpha,n}(z) := 
\frac{1}{\sigma_\alpha(n)}L_n^{(\alpha)}(z),\quad n{\ge0},
\ee
where 
\be\label{eq:sigma_a}
\sigma_\alpha(n):=\sqrt{L_n^{(\alpha)}(0)} = \binom{n+\alpha}{n}^{1/2},\qquad n{\ge0},\quad \alpha>-1,
\ee
and $L_n^{(\alpha)}$ are the Laguerre polynomials
\cite[Section 5.1]{sz}:
\be\label{eq:laguerpol}
L_n^{(\alpha)}(z) = \frac{\E^{z}z^{-\alpha}}{n!}\frac{d^n}{dz^n}\E^{-z} z^{n+\alpha} 
=\binom{n+\alpha}{n} \sum_{k=0}^n\frac{(-n)_k}{(\alpha+1)_k\,k!}\,z^k,\quad n{\ge0}.
\ee
Orthogonality relations for $P_{\alpha,n}$ are given by (see \cite[(5.1.1)]{sz})
\be\label{eq:lagorth}
\frac1{\Gamma(\alpha+1)}
\int_{0}^\infty P_{\alpha,n}(\lambda) P_{\alpha,k}(\lambda) \E^{-\lambda}\lambda^\alpha\, d\lambda 
= \delta_{n,k},\quad n,k\in\Z_{\ge0}.
\ee
Therefore, the probability measure
\be
\rho_\alpha(d\lambda) = \frac{1}{\Gamma(\alpha+1)}\id_{\R_{>0}}(\lambda)\E^{-\lambda}\lambda^\alpha d\lambda
\ee
is the spectral measure of $H_\alpha$, that is, $H_\alpha$ is unitarily equivalent to a multiplication operator in $L^2(\R_{>0};\rho_\alpha)$. Indeed, the map $\cF_\alpha\colon \ell^2(\Z_{\ge 0})\to L^2(\R_{>0};\rho_\alpha)$ defined by
\be\label{eq:Falpha}
(\cF_\alpha f)(\lambda):=\sum_{n\ge 0}f_n P_{\alpha,n}(\lambda) ,\quad \lambda>0,
\ee
for all $f\in \ell^2_c(\Z_{\ge 0})$, extends to an isometric isomorphism. Its inverse 
is given by
\begin{align*}
(\cF_\alpha^{-1}F)_n = \frac{1}{\Gamma(\alpha+1)}\int_{0}^\infty F(\lambda) P_{\alpha,n}(\lambda)\E^{-\lambda}\lambda^\alpha\, d\lambda,\quad n{\ge 0},
\end{align*}
for every $F\in L^2_c(\R_{>0};\rho_\alpha)$. Then 
\be\label{eq:Halpha=M}
H_\alpha = \cF_\alpha^{-1} M_\alpha \cF_\alpha,
\ee
 where $M_\alpha$ is the multiplication operator 
\begin{align*}
M_\alpha\colon F(\lambda)\mapsto \lambda F(\lambda).
\end{align*}
acting in the Hilbert space $L^2(\R_{>0};\rho_\alpha)$. This in particular implies that $H_\alpha$ is a positive operator and its spectrum $\sigma(H_\alpha)$ coincides with $[0,\infty)$. Moreover, $\sigma(H_\alpha)$ is purely absolutely continuous of multiplicity $1$.

The Stieltjes transform of $\rho_\alpha$, which is usually called the {\em Weyl function} (or $m$-function) of $H_\alpha$, is given by
\be\label{eq:malpha}
m_\alpha(z)  = \frac{1}{\Gamma(\alpha+1)}\int_0^{+\infty} \frac{\E^{-\lambda}\lambda^\alpha}{\lambda-z}d\lambda=
\E^{-z} { E}_{1+\alpha}(-z) ,\quad z\in\C\setminus \R_{\ge0},
\ee
where 
\be\label{eq:Ei}
E_p(z) := z^{p-1} \int_z^\infty \E^{-t} t^{-p} dt = z^{p-1} \Gamma(1-p,z)
\ee
denotes the principal value of the generalized exponential integral \dlmf{8.19.2} and $\Gamma(s,z)$ is the incomplete Gamma function \dlmf{8.2.2}. 
Note that
\be\label{eq:mat0}
m_\alpha(-0):=\lim_{x\downarrow 0}m_\alpha(-x) = \begin{cases} 1/\alpha, & \alpha>0, \\+\infty, & \alpha\in (-1,0]. \end{cases}
\ee

Next, let us define the polynomials of the second kind (see \cite{akh}):
\be\label{eq:seckind}
Q_{\alpha,n}(z) := \frac{1}{\Gamma(\alpha+1)}\int_0^\infty \frac{P_{\alpha,n}(z)-P_{\alpha,n}(\lam)}{z-\lam} \E^{-\lam}\lambda^{\alpha}d\lam,\qquad n\ge 0,
\ee
where $Q_{\alpha,0}(z)\equiv 0$ and $Q_{\alpha,1}(z)\equiv \frac{1}{\sqrt{\alpha+1}}$.
Then $u:= (Q_{\alpha,n}(z))_{n{\ge0}}$ satisfies $(\tau_\alpha u)_n = zu_n$
for all $n\ge 1$. 
Notice that for all $z\in \C\setminus \R_{\ge0}$ the linear combination
\be\label{eq:weylsol}
\Psi_{\alpha,n}(z):= Q_{\alpha,n}(z) + m_\alpha(z) P_{\alpha,n}(z),\quad n{\ge0}, 
\ee
also known as the {\em Weyl solution} in the Jacobi operators context, satisfies 
\be
(\Psi_{\alpha,n}(z))_{n\ge 0} \in \ell^2(\Z_{\ge0})
\ee 
for all $z\in \C\setminus \R_{\ge0}$. 
In particular, this provides us with the explicit expression of the resolvent of $H_\alpha$ (actually, with its Green's function)
\be\label{eq:H0resolv}
G_\alpha(z;n,m):= \big\langle (H_\alpha - z)^{-1}\delta_n ,\delta_m \big\rangle_{\ell^2}  = \begin{cases} P_{\alpha,n}(z) \Psi_{\alpha,m}(z), & n\le m,\\ P_{\alpha,m}(z) \Psi_{\alpha,n}(z), & n\ge m. \end{cases}
\ee

\begin{lemma}\label{lem:Greenat0}
For $n,m\in\Z_{\ge 0}$,
\be\label{eq:Greenat0}
G_\alpha(-0;n,m) := \lim_{x\uparrow-0} G_\alpha(x;n,m) = \begin{cases} \frac{1}{\alpha}\frac{\sigma_{\alpha}(\min(n,m))}{\sigma_{\alpha}(\max(n,m))}, & \alpha>0, \\[2mm] +\infty, & \alpha\in (-1,0]. \end{cases}
\ee
\end{lemma}

\begin{proof}
Since $H_\alpha$ is self-adjoint, $G_\alpha(x;n,m) = G_\alpha(x;m,n)$ for any $x<0$, so suppose $n\le m$. By \eqref{eq:H0resolv},
\be\label{eq:Greenat0_01}
G_\alpha(-0;n,m) = \sigma_\alpha(n)\big(Q_{\alpha,m}(0) + \sigma_\alpha(m) m_\alpha(-0)\big),
\ee
and hence \eqref{eq:mat0} implies \eqref{eq:Greenat0} for $\alpha\in(-1,0]$.  
If $\alpha>0$, we get by using \eqref{eq:seckind},
\begin{align*}
Q_{\alpha,m}(0) 
&= \frac{1}{\Gamma(\alpha+1)\sigma_\alpha(m)}\int_0^\infty \frac{L^{(\alpha)}_m(\lam) - L^{(\alpha)}_m(0)}{\lam} \E^{-\lam}\lambda^{\alpha}d\lam \\
& = \frac{1}{\Gamma(\alpha+1)\sigma_\alpha(m)}\int_0^\infty \big(\sum_{k=0}^m L^{(\alpha-1)}_k(\lam) - L^{(\alpha)}_m(0)\big) \E^{-\lam}\lambda^{\alpha-1}d\lam \\
& = \frac{1}{\Gamma(\alpha+1)\sigma_\alpha(m)}\int_0^\infty \big( L^{(\alpha-1)}_0(\lam) - L^{(\alpha)}_m(0)\big) \E^{-\lam}\lambda^{\alpha-1}d\lam \\
& = \frac{1}{\Gamma(\alpha+1)\sigma_\alpha(m)}\Gamma(\alpha) (1 - \sigma_\alpha(m)^2) \\
& = \frac{1 - \sigma_\alpha(m)^2}{\alpha\sigma_\alpha(m)}.
\end{align*}
Here in the second line we used \dlmf{18.18.37} and then orthogonality of the Laguerre polynomials \eqref{eq:lagorth}. It remains to plug the last expression into \eqref{eq:Greenat0_01}.
\end{proof}

\section{The quadratic form}\label{sec:III}

Let us consider the quadratic form corresponding to the operator $H_\alpha$:
\begin{align}\label{eq:gta0}
\gt_\alpha^0[u] & : = \langle H_\alpha u, u \rangle_{\ell^2}, & u\in \dom(\gt_\alpha^0) & := \dom(H_\alpha).
\end{align}
This form is positive since so is $H_\alpha$. Since $H_\alpha$ is self-adjoint, $\gt_\alpha^0$ is closable and its closure $\gt_\alpha$ is explicitly given by
\begin{align}\label{eq:gta}
\gt_\alpha[u] & = \|\sqrt{H_\alpha} u\|^2_{\ell^2}, & u\in \dom(\gt_\alpha) & = \dom(\sqrt{H_\alpha}),
\end{align}
where $\sqrt{H_\alpha}$ denotes the positive self-adjoint square root of $H_\alpha$.

\begin{lemma}\label{lem:form}
The domain $\dom(\gt_\alpha)$ of $\gt_\alpha$ does not depend on $\alpha$ and consists of those $u\in \ell^2(\Z_{\ge 0})$ for which the series
\be
\sum_{n\ge 0} (n+1)\big| u_n - u_{n+1}\big|^2
\ee
is finite. Moreover, for every $u\in\dom(\gt_\alpha)$ the form $\gt_\alpha$ admits the representation 
\be\label{eq:gtafact}
\gt_\alpha[u] = \sum_{n\ge 0} \big| \sqrt{n+\alpha+1} u_n - \sqrt{n+1} u_{n+1}\big|^2.
\ee
\end{lemma}

\begin{proof}
Observe that the matrix \eqref{eq:H0} admits ``the string factorization" (see, e.g., \cite[Appendix]{akh},  \cite[\S 13]{kakr74}, \cite[\S 3]{ek18}): 
\begin{align}\label{eq:Stifact}
h_{n,n}^{(\alpha)} & = \frac{1}{l_{\alpha}(n)}\left(\frac{1}{\omega_{\alpha}(n-1)} + \frac{1}{\omega_{\alpha}(n)}\right), & 
h_{n,n+1}^{(\alpha)} & = \frac{-1}{\omega_{\alpha}(n) \sqrt{l_{\alpha}(n)l_{\alpha}(n+1)}},
\end{align}
 where $\frac{1}{\omega_{\alpha}(-1)}:=0$  and 
\begin{align}\label{eq:la}
l_{\alpha}(n) & = |P_{\alpha,n}(0)|^2 = \sigma_\alpha(n)^2 = \frac{(\alpha+1)_n}{n!}, \\ 
\omega_{\alpha}(n) & =  -\frac{1}{h_{n,n+1}^{(\alpha)}\sqrt{l_{\alpha} (n)l_{\alpha}(n+1)}} 
 = \frac{n!}{(\alpha+1)_{n+1}}, \quad n\ge 0. \label{eq:wa}
\end{align}
Therefore, the Jacobi matrix \eqref{eq:H0} can be (at least formally) written as
\be\label{eq:Hfactor}
H_\alpha = \cL_\alpha^{-1} (I-\cS) \cW_\alpha^{-1} (I-\cS^\ast) \cL_\alpha^{-1},
\ee
where $\cW_\alpha$ and $\cL_\alpha$ are the multiplication operators 
\begin{align}\label{eq:La}
\cL_\alpha\colon & (u_n)_{n\ge 0} \mapsto (\sigma_\alpha(n)u_n)_{n\ge 0}, & \cW_\alpha\colon & (u_n)_{n\ge0} \mapsto (\omega_{\alpha}(n)u_n)_{n\ge0},
\end{align}
$\cS$ is the shift operator $\cS\colon (u_n)_{n\ge 0}  \mapsto   (u_{n-1})_{n\ge 0}$ with the standard convention $u_{-1}:=0$, and $\cS^\ast$  is the backward shift, $\cS^\ast \colon (u_n)_{n\ge 0}  \mapsto   (u_{n+1})_{n\ge 0}$. 
The representation \eqref{eq:Hfactor} immediately implies
\begin{align*}
\gt_\alpha^0[u]  = \langle H_\alpha u, u \rangle_{\ell^2} & = \|\cW_\alpha^{-1/2} (I-\cS^\ast) \cL_\alpha^{-1} u\|_{\ell^2}^2 \\
& = \sum_{n\ge 0} \frac{1}{\omega_{\alpha}(n)}\Big|\frac{u_n}{\sigma_\alpha(n)} - \frac{u_{n+1}}{\sigma_\alpha(n+1)}\Big|^2\\
& = \sum_{n\ge 0} \big| \sqrt{n+\alpha+1} u_n - \sqrt{n+1} u_{n+1}\big|^2,
\end{align*}
for every $u\in \ell^2_c(\Z_{\ge 0})$. 

Consider now the maximally defined form $\gt_\alpha$, i.e., $\gt_\alpha$ is defined by the RHS in \eqref{eq:gtafact} on sequences $u\in \ell^2(\Z_{\ge 0})$ for which the RHS in \eqref{eq:gtafact} is finite. It is standard to show that this form is positive and closed in $\ell^2(\Z_{\ge 0})$. However, $\gt_\alpha$ is clearly an extension of the pre-minimal form $\gt^0_\alpha$. However, the maximally defined operator $H_\alpha$ is self-adjoint and hence $\gt^0_\alpha$ admits a unique closed extension. Thus $\gt_\alpha = \overline{\gt_\alpha^0}$. 

Finally, to show that $\dom(\gt_\alpha) = \dom(\gt_0)$ for all $\alpha>-1$ it suffices to notice that 
\begin{align*}
\big|\sqrt{n+\alpha+1} - \sqrt{n+1}\big| = \frac{|\alpha|}{\sqrt{n+\alpha+1} + \sqrt{n+1}} \le \frac{|\alpha|}{\sqrt{n+1}}
\end{align*}
for all $n\ge0$ and $\alpha>-1$. 
\end{proof}

\begin{remark}\label{rem:string}
It was observed by Mark Krein \cite[\S 13]{kakr74} that spectral theory of Jacobi matrices admitting factorization \eqref{eq:Stifact} can be included into the spectral theory of Krein strings. Indeed, setting 
\begin{align*}
x_{-1}:=0,\qquad x_n = x_\alpha(n):=  \sum_{k=0}^{n}l_\alpha(k) = \sum_{k=0}^{n}\frac{(\alpha+1)_k}{k!},\qquad n\ge 0, 
\end{align*}
and introducing a measure $\omega_\alpha$ on $[0,\infty)$ by 
\begin{align*}
\omega([0,x)) = \sum_{x_n<x} \omega_\alpha(n) = \sum_{x_n<x} \frac{n!}{(\alpha+1)_{n+1}}, \quad x\ge 0,
\end{align*}  
the difference equation $\tau_\alpha u = z u$ describes small oscillations of a string of infinite length which carries only point masses $\omega_\alpha(n)$ at $x_n$. Moreover, $\tau_\alpha u = z u$ can be turned into the string spectral problem (the so-called Krein--Stieltjes string)
\begin{align*}
- y'' = z\omega_\alpha y\quad \text{on}\quad [0,\infty).
\end{align*} 
\end{remark}

\begin{corollary}
The operator $H_\alpha$ generates a positivity preserving semigroup $\E^{-tH_\alpha}$, $t>0$. Moreover, for $\alpha=0$, the corresponding semigroup is Markovian.
\end{corollary}

\begin{proof}
By the first Beurling--Deny criterion (see \cite[Theorem XIII.50]{rsIV}), it suffices to notice that
\begin{align*}
\langle H_\alpha |u|,|u|\rangle_{\ell^2} \le \langle H_\alpha u,u\rangle_{\ell^2}
\end{align*}
in view of \eqref{eq:gtafact}. Here $|u|: = (|u_n|)_{n\ge 0}$. 

If $\alpha=0$, then by Lemma \ref{lem:form}
\begin{align*}
\gt_0[u] = \sum_{n\ge 0} (n+1)\big|  u_n -  u_{n+1}\big|^2
\end{align*}
for all $u\in \dom(\gt_0)$. Suppose additionally that $u\ge 0$, that is, $u_n\ge 0$ for all $n\ge 0$. It is straightforward to check that $\min(u,\id)$ also belongs to $\dom(\gt_0)$ and, moreover,
\begin{align*}
\gt_0[\min(u,\id)] \le \gt_0[u].
\end{align*}
By the second Beurling--Deny criterion (see \cite[Theorem XIII.51]{rsIV}), $\E^{-t H_0}$, $t>0$ extends to a contraction on $\ell^p$ for each $p\in[1,\infty]$. This implies that it is Markovian and the corresponding quadratic form $\gt_0$ is a Dirichlet form \cite[\S 1.4]{fuk10}.
\end{proof}

\begin{remark}\label{rem:DirForm}
For $\alpha\neq 0$ the form $\gt_\alpha$ is not a Dirichlet form. Indeed, by Lemma \ref{lem:form}, for each  $0\le u\in \dom(\gt_\alpha) = \dom(\gt_0)$, we get $\min(u,\id) \in \dom(\gt_0) = \dom(\gt_\alpha)$. However, one can construct a positive $u\in \dom(\gt_\alpha)$ such that $\gt_\alpha[\min(u,\id)] > \gt_\alpha[u]$. Therefore, the semigroup 
$\E^{-t H_\alpha}$, $t>0$ is not Markovian if $\alpha\neq 0$.
\end{remark}

In fact, the form $\gt_\alpha$ is closely connected with the Dirichlet form and this form would be important in our analysis. Consider the weighted space $\ell^2(\Z_{\ge 0};\sigma_\alpha^2)$. The multiplication operator $\cL_\alpha$ given by \eqref{eq:La} defines an isometric isomorphism from $\ell^2(\Z_{\ge 0};\sigma_\alpha^2)$ onto $\ell^2(\Z_{\ge 0})$. Consider the operator $\wt{H}_\alpha$ defined on $\ell^2(\Z_{\ge 0};\sigma_\alpha^2)$ by 
\be\label{eq:wtHalpha}
\wt{H}_\alpha = \cL_\alpha^{-1} {H}_\alpha \cL_\alpha.
\ee  
The corresponding difference expression is given by
\begin{align}\label{eq:StringDiff}
\begin{split}
(\wt\tau_\alpha u)_n & = - n\,u_{n-1} + (2n+1 +\alpha)u_n - (n+1+\alpha)\,u_{n+1} \\
& = \frac{1}{\sigma_\alpha(n)^2} \sum_{k\ge 0\colon |k-n|=1} \frac{u_n - u_k}{\omega_\alpha(\min(n,k))},\quad n\ge 0.
\end{split}
\end{align}
Then it is easy to check that the quadratic form $\wt{\gt}_\alpha$ is simply given by
\begin{align}\label{eq:wtform}
\begin{split}
\wt{\gt}_\alpha[u] =  \langle \wt{H}_\alpha u, u \rangle_{\ell^2(\sigma_\alpha^2)} &= \|\cW_\alpha^{-1/2}(I-\cS^\ast) u\|^2_{\ell^2}\\
& =  \sum_{n\ge 0} \frac{1}{\omega_\alpha(n)}|{u_n} - {u_{n+1}}|^2 \\
&= \sum_{n\ge 0} \frac{(\alpha+1)_{n+1}}{n!}|{u_n} - {u_{n+1}}|^2 
\end{split}
\end{align}
for every $u\in\ell^2_c$. The closure of this form is a regular Dirichlet form in $\ell^2(\Z_{\ge 0};\sigma_\alpha^2)$. 

\begin{corollary}\label{cor:DirForm}
Let $\alpha>-1$ and $\wt{H}_\alpha$ be the operator \eqref{eq:wtHalpha} acting in $\ell^2(\Z_{\ge 0};\sigma_\alpha^2)$. Then $\wt{H}_\alpha$ is Markovian, that is, the corresponding semigroup 
\be\label{eq:semiWT}
\E^{-t\wt{H}_\alpha}= \cL_\alpha^{-1} \E^{-t {H}_\alpha} \cL_\alpha,\qquad t>0,
\ee
 is positivity preserving and $\ell^\infty$ contractive.
\end{corollary}

\begin{remark}\label{rem:RW}
The first line in \eqref{eq:StringDiff} shows that $\wt{H}_\alpha$ generates a birth-and-death process on $\Z_{\ge 0}$ (see \cite[Chapter 17.5]{feller}), however, the second line connects $\wt{H}_\alpha$ with a continuous-time random walk (a simple Markov chain) on $\Z_{\ge 0}$ (see \cite{fuk10, klwBook}). The latter is not at all surprising since their connections with the Stieltjes moment problem and Krein--Stieltjes strings is widely known (see, e.g., \cite{kac05}).
\end{remark}

As a by-product of the factorization \eqref{eq:Stifact} we arrive at the following continued fraction representation of the exponential integral \eqref{eq:Ei} and the incomplete gamma function $\Gamma(s,z)$. We do not need this formula for our future purposes, however, it is so beautiful that we decided to include it together with a short proof.

\begin{corollary}\label{cor:EiCF}
Let $\alpha>-1$. Then 
\be\label{eq:EiCF}
E_{1+\alpha}(z) =  z^\alpha\Gamma(-\alpha,z) = \cfrac{ \E^{-z}}{z +  \cfrac{\alpha+1 }{1 + \cfrac{1}{z +  \cfrac{\alpha+2}{1 +\cfrac{2}{z + \cfrac{\alpha+3}{ 1 + \cfrac{3}{\ddots}}}}}}},
\ee
which converges for all $z\in \C\setminus (-\infty,0]$. 
\end{corollary}

\begin{proof}
The string factorization \eqref{eq:Stifact} implies the following Stieltjes continued fraction representation of the Weyl function $m_\alpha$ (see, e.g., \cite[\S 13]{kakr74}, \cite{sti}, \cite[\S 3]{ek18}):
\be
m_\alpha(z) = \cfrac{1}{- z\, l_{\alpha}(0) +  \cfrac{1}{\omega_{\alpha}(0) + \cfrac{1}{- z\,l_{\alpha}(1) +  \cfrac{1}{\omega_{\alpha}(1) +\cfrac{1}{\ddots}}}}}\, ,
\ee
which converges locally uniformly in $\C\setminus\R_{\ge 0}$ (this follows from the self-adjointness of $H_\alpha$, see, e.g., \cite{akh}, \cite[\S 3]{ek18}). Taking into account  \eqref{eq:la} and \eqref{eq:wa} and noting that
\begin{align*}
(\alpha+n)\frac{\sigma_{\alpha}(n-1)^2}{\sigma_{\alpha}(n)^2} = (\alpha+n) \frac{(\alpha+1)_{n-1}}{(n-1)!} \frac{n!}{(\alpha+1)_n} = n
\end{align*}
for all $n>0$, we arrive at \eqref{eq:EiCF}.
\end{proof}

\begin{remark}
The continued fraction expansion \eqref{eq:EiCF} of the exponential integral is by no means new and the case $\alpha=0$ can already be found in the work of Stieltjes \cite[Chapter IX]{sti} (see also \dlmf{8.9.2}, \dlmf{8.19.17}, \cite[\S 6.7.1]{gst} and \cite[(14.1.6)]{cuyt}).
\end{remark}

\section{The heat semigroup}\label{sec:IV}
In this section we look at the one-dimensional discrete heat equation
\begin{align} \label{eq:heat}
   \dot \psi(t,n) & = - H_\alpha \psi(t,n), & (t,x) & \in\R_{>0}\times \Z_{\ge0},
\end{align}
associated with the Laguerre operator $H_\alpha$, as well as at the closely related heat equation
\begin{align} \label{eq:heatWT}
   \dot \psi(t,n) & = - \wt{H}_\alpha \psi(t,n), & (t,x) & \in\R_{>0}\times \Z_{\ge0},
\end{align}
associated with the operator $\wt{H}_\alpha$ defined in the previous section. We set
 \be\label{eq:heatMatrix}
 \E^{- tH_\alpha}(n,m) := \langle \E^{- tH_\alpha}\delta_n, \delta_m\rangle_{\ell^2},\qquad (n,m)\in\Z_{\ge 0}\times\Z_{\ge 0},
 \ee
and 
 \be\label{eq:heatWTMatrix}
 \E^{- t\wt{H}_\alpha}(n,m) := \langle \E^{- t\wt{H}_\alpha}\delta_n, \delta_m\rangle_{\ell^2(\sigma_\alpha^2)},\qquad (n,m)\in\Z_{\ge 0}\times\Z_{\ge 0}.
 \ee
 Notice that \eqref{eq:heatWTMatrix} does not coincide with the matrix representation of $ \E^{- t\wt{H}_\alpha}$ in an orthonormal basis (if $\alpha\neq 0$) and we defined it this way in order to write the heat kernel of $\wt{H}_\alpha$ in the form familiar in the continuous context, that is, in the form 
\be\label{eq:heatWTaction}
(\E^{-t\wt{H}_\alpha} u)_n = \sum_{m\ge 0}   \E^{- t\wt{H}_\alpha}(n,m) u_m \sigma_\alpha(m)^2.
\ee

\subsection{Connection with Jacobi polynomials}

We begin by establishing a connection between the discrete Laguerre operators and Jacobi polynomials, which follows from the fact that the Laplace transform of a product of two Laguerre polynomials is expressed by means of a terminating hypergeometric series.

\begin{theorem}\label{thm:explicit}
Let $\alpha>-1$. The kernel of the heat semigroup $\E^{- tH_\alpha}$ is given by 
\begin{multline}\label{eq:explicit}
\E^{- tH_\alpha}(n,m)=\E^{- tH_\alpha}(m,n)\\
= \frac {1}{(1+ t)^{1+\alpha}}\left(\frac{t-1}{t+1}\right)^n\left(\frac{t}{t+1}\right)^{m-n}\,
\frac{\sigma_\alpha(m)}{\sigma_\alpha(n)}\ 
P_n^{(\alpha,m-n)}\left( \frac{t^2+1}{t^2-1}\right)
\end{multline} 
for all $n$, $m\in\Z_{\ge0}$.
\end{theorem} 

\begin{proof}
Taking into account \eqref{eq:Halpha=M}, \eqref{eq:Falpha} and then \eqref{eq:P=L}, we get
\begin{align*}
\langle \E^{-tH_\alpha} \delta_n,\delta_m\rangle_{\ell^2} &= \langle \cF^{-1}_\alpha \E^{- tM_\alpha} \cF_\alpha \delta_n,\delta_m\rangle_{\ell^2} \\ &= \langle  \E^{- tM_\alpha} \cF_\alpha \delta_n,\cF_\alpha\delta_m\rangle_{L^2(\rho_\alpha)} \\
& = \frac{1}{\sigma_\alpha(n)\sigma_\alpha(m)\Gamma(\alpha+1)}\int_0^\infty \E^{-(1+ t)\lambda} L_n^{(\alpha)}(\lambda)L_m^{(\alpha)}(\lambda) \lambda^\alpha d\lambda,
\end{align*}
for $n,m\in\Z_{\ge0}$. Thus, every element of the kernel of the operator $\E^{- tH_\alpha}$ is the Laplace transform of a product of two Laguerre polynomials and hence we get (see \cite[(4.11.35)]{erd} and \dlmf{15.8.7}): 
\be\label{K1}
\frac{\E^{- tH_\alpha}(n,m)}{\sigma_\alpha(n)\sigma_\alpha(m)}
	 = \frac{ t^{n+m}}{(1+t)^{n+m+\alpha+1}}\,  \hyp21{-n,-m}{\alpha+1}{\frac{1}{t^2}},
\ee
By Euler's transformation \dlmf{15.8.1},
\begin{align*}
\hyp21{-n,-m}{\alpha+1}{\frac{1}{t^2}} = \left(\frac{t^2-1}{t^2}\right)^{n}\hyp21{-n,\alpha+m+1}{\alpha+1}{\frac{1}{1-t^2}}.
\end{align*}
Hence by \eqref{eq:normal} and \eqref{eq:01},
 \eqref{K1} implies \eqref{eq:explicit}
\end{proof}

\begin{remark}
Formula \eqref{eq:explicit} can be derived from \cite[Theorem 3.1]{kkt18} by analytic continuation. Namely, in \cite{kkt18}, it was shown that the kernel of the unitary evolution $\E^{-\I tH_\alpha}$ is given by  
\begin{multline}\label{eq:explicitSchr}
\E^{-\I tH_\alpha}(n,m)=\E^{-\I tH_\alpha}(m,n)\\
= \frac {1}{(1+\I t)^{1+\alpha}}\left(\frac{t+\I}{t-\I}\right)^n\left(\frac{t}{t-\I}\right)^{m-n}\,
\frac{\sigma_\alpha(m)}{\sigma_\alpha(n)}\
P_n^{(\alpha,m-n)}\left( \frac{t^2-1}{t^2+1}\right)
\end{multline} 
for all $n$, $m\in\Z_{\ge 0}$. Replacing $\I t$ by $t$ in \eqref{eq:explicitSchr}, we end up with \eqref{eq:explicit}.
\end{remark}

We collect some special cases explicitly for later use.

\begin{corollary}\label{cor:cases}
 \begin{enumerate}[label=(\roman*), ref=(\roman*), leftmargin=*, widest=iii] 
\item In the case $n=0$ we have 
\be\label{eq:erd1}
\E^{- tH_\alpha}(0,m)  = 
\frac { \sigma_\alpha(m)}{(1+ t)^{1+\alpha}}\left(\frac{t}{t+1}\right)^m,\quad m\in\Z_{\ge0}.
\ee
\item In the case $n=1$ we have for $m\in\Z_{\ge 1}$
\be\label{eq:erd2}
\E^{- tH_\alpha}(1,m)=
\frac {1}{(1+ t)^{1+\alpha}}\left(\frac{ t}{t+1}\right)^{m-1}
\frac{(1+\alpha)t^2+ m}{(t+1)^2}\,\frac{\sigma_\alpha(m)}{\sigma_\alpha(1)}.
\ee
\item
In the case $n=m$ we have
\be\label{eq:erd3}
\E^{- tH_\alpha}(m,m) = \frac {1}{(1+ t)^{1+\alpha}}\left(\frac{t-1}{t+1}\right)^m P_m^{(\alpha,0)}\left( \frac{t^2+1}{t^2-1}\right), \quad m\in\Z_{\ge0}.
\ee
\end{enumerate}
\end{corollary}

\begin{proof}
Just observe that
\begin{equation*}
P_0^{(\alpha,m)}(z) = 1, \quad P_1^{(\alpha,m-1)}(z) = -m+(m+1+\alpha)\frac{z+1}{2}.\qedhere
\end{equation*}
\end{proof}

Taking into account \eqref{eq:semiWT}, one also easily derives the explicit expression for the heat kernel of \eqref{eq:heatWT}. Since $H_\alpha$ and $\wt{H}_\alpha$ are unitarily equivalent, the matrix of $\E^{-t\wt{H}_\alpha}$ in the orthonormal basis $(\cL_\alpha^{-1}\delta_n)_{n\ge 0}$ coincides with that of $\E^{-t H_\alpha}$. However, our definition of \eqref{eq:heatWTMatrix} is slightly different and in fact \eqref{eq:heatMatrix}--\eqref{eq:heatWTMatrix} gives
\be
\E^{- t\wt{H}_\alpha}(n,m) = \frac{\E^{- t H_\alpha}(n,m)}{\sigma_\alpha(n)\sigma_\alpha(m)},\qquad (n,m)\in\Z_{\ge 0}\times \Z_{\ge 0}.
\ee 

\begin{remark}\label{rem:heat=Meixner}
The heat kernel can be expressed in terms of Meixner polynomials \dlmf{18.20.7}:
\be\label{def:Meixner}
M_n(x;\beta,c) := \hyp21{-n,-x}{\beta}{1-\frac{1}{c}}.
\ee
Thus \eqref{K1} reads
\begin{align}\label{eq:heat=Meixner}
\E^{- t\wt{H}_\alpha}(n,m) = \frac{\E^{- t H_\alpha}(n,m)}{\sigma_\alpha(n)\sigma_\alpha(m)} 
= \frac{1}{(1+t)^{\alpha+1}}\Big(\frac{t}{t+1}\Big)^{n+m} M_n\Big(m;1+\alpha,\frac{t^2}{t^2-1}\Big).
\end{align}
\end{remark}

\subsection{Heat semigroup estimates}

Our next aim is to obtain uniform estimates on the elements of the heat kernel.
First observe the following simple bounds.

\begin{lemma}\label{lem:diagest01}
Let $\alpha>-1$ and $n,m \ge 0$. Then 
\be
(1+t)^{1+\alpha}\E^{- tH_\alpha}(n,m) = \sigma_\alpha(n)\sigma_\alpha(m) +\OO(t^{-1}) 
\ee
as $t\to \infty$, and 
\be
\E^{- tH_\alpha}(n,m) = \binom{\max(n,m)}{\min(n,m)}\frac{\sigma_\alpha(\max(n,m))}{\sigma_\alpha(\min(n,m))}t^{|n-m|}(1+\OO(t)) 
\ee
as $t\to +0$.
\end{lemma}

\begin{proof}
Immediately follows from \eqref{eq:normal} and \eqref{eq:normal_b} (see also \eqref{K1}).
\end{proof}

The latter indicates that one can hope for the following uniform estimate 
\be
 \sup_{n,m\in\Z_{\ge 0}} \frac{|\E^{-t H_\alpha}(n,m)|}{\sigma_\alpha(n)\sigma_\alpha(m)} \le \frac{C}{(1+t)^{1+\alpha}}
\ee
for all positive $t>0$, where $C=C(\alpha)>0$ may depend on $\alpha$. 
The next statement confirms the desired bound for $\alpha\ge 0$. 

\begin{theorem}\label{thm:HeatEst}
If  $\alpha\ge 0$, then
\be\label{eq:HeatEst1}
\|\E^{-t H_\alpha}\|_{\ell^1(\sigma_\alpha)\to \ell^\infty(\sigma_\alpha^{-1})} = \frac{1}{(1+t)^{1+\alpha}},\quad t>0.
\ee
If $\alpha\in(-1,0)$, then
\be\label{eq:HeatEst2}
\|\E^{-t H_\alpha}\|_{\ell^1(\sigma_\alpha)\to \ell^\infty(\sigma_\alpha^{-1})} \ge \frac{1}{(1+t)^{1+\alpha}},\quad t>0.
\ee
\end{theorem}

\begin{proof}
By definition, for $t>0$ we get
\begin{align*}
\|\E^{-t H_\alpha}\|_{\ell^1(\sigma_\alpha)\to \ell^\infty(\sigma_\alpha^{-1})} = \sup_{n,m\in\Z_{\ge 0}} \frac{|\E^{-t H_\alpha}(n,m)|}{\sigma_\alpha(n)\sigma_\alpha(m)} = \|\E^{-t \wt{H}_\alpha}\|_{\ell^1(\sigma_\alpha^2)\to \ell^\infty}.
\end{align*}
Using \eqref{eq:erd1}, we get 
\begin{align*}
\|\E^{-t H_\alpha}\|_{\ell^1(\sigma_\alpha)\to \ell^\infty(\sigma_\alpha^{-1})} \ge \E^{-t H_\alpha}(0,0) = \frac{1}{(1+t)^{1+\alpha}}
\end{align*}
for all $\alpha>-1$ and $t>0$.  
Thus, it remains to show that
\begin{align*}
\|\E^{-t H_\alpha}\|_{\ell^1(\sigma_\alpha)\to \ell^\infty(\sigma_\alpha^{-1})} \le \frac{1}{(1+t)^{1+\alpha}},\quad t>0,
\end{align*}
when $\alpha\ge 0$. By Corollary \ref{cor:DirForm}, $\E^{-t \wt{H}_\alpha}$ is positivity preserving and $\ell^\infty$ contractive and hence
\begin{align*}
0 < \E^{-t \wt{H}_\alpha}(n,n) \le 1
\end{align*}
for all $n\ge 0$ and $t>0$. Therefore, 
\begin{align*}
|\E^{-t \wt{H}_\alpha}(n,m)|^2 \le \E^{-t \wt{H}_\alpha}(n,n)\cdot \E^{-t \wt{H}_\alpha}(m,m) \le \max \Big(\E^{-t \wt{H}_\alpha}(n,n),\E^{-t \wt{H}_\alpha}(m,m) \Big),
\end{align*}
which immediately implies 
\begin{align*}
\|\E^{-t H_\alpha}\|_{\ell^1(\sigma_\alpha)\to \ell^\infty(\sigma_\alpha^{-1})} =  \sup_{n \ge 0} \frac{ \E^{-t H_\alpha}(n,n)}{\sigma_\alpha(n)^2} =  \sup_{n \ge 0}  \E^{-t \wt{H}_\alpha}(n,n). 
\end{align*}
Thus, \eqref{eq:erd3} implies that it suffices to prove the inequality
\begin{align*}
\left(\frac{t-1}{t+1}\right)^n P_n^{(\alpha,0)}\left( \frac{t^2+1}{t^2-1}\right) \le \binom{n+\alpha}{n}
\end{align*}
for all $n\ge 0$ and $t>0$. First, using the Rodrigues formula \eqref{K5}, we get
\begin{align*}
\left(\frac{t-1}{t+1}\right)^n P_n^{(\alpha,0)}\left( \frac{t^2+1}{t^2-1}\right) 
 &=\left(\frac{t-1}{t+1}\right)^n \sum_{k=0}^n \binom{n+\alpha}{n-k} \binom{n}{k} \left(\frac{1}{t^2-1}\right)^k\left(\frac{t^2}{t^2-1}\right)^{n-k} \\
 &= \frac{1}{(t+1)^{2n}} \sum_{k=0}^n \binom{n+\alpha}{n-k} \binom{n}{k} t^{2(n-k)}\\ 
 &= \frac{1}{(t+1)^{2n}} \sum_{k=0}^n \binom{n+\alpha}{k} \binom{n}{k} t^{2k}.
\end{align*}
Since
\begin{align*}
(t+1)^{2n} = \sum_{k=0}^{2n} \binom{2n}{k} t^{k} 
 > \sum_{k=0}^{n}\binom{2n}{2k}t^{2k}
\end{align*}
for all $t>0$ and $n\in\Z_{\ge 1}$, it suffices to show that
\be\label{eq:binomIneq}
\binom{n+\alpha}{k} \binom{n}{k} \le \binom{n+\alpha}{n} \binom{2n}{2k}
\ee
for all $n,k \in \Z_{\ge 0}$ with $k<n$. 
Using \eqref{K6}, it is easy to observe that 
 \eqref{eq:binomIneq} is equivalent to the following inequality
\begin{align*}
\prod_{j=0}^{k-1} \frac{n+\alpha-j}{k-j} \le \prod_{j=0}^{n-1} \frac{n+\alpha-j}{n-j} \prod_{j=0}^{k-1} \frac{2n-2j -1}{2k - 2j - 1}.
\end{align*}
However, the latter holds exactly when 
\begin{align*}
\prod_{j=0}^{k-1} \frac{2k-2j - 1}{k-j} \le  \prod_{j=0}^{k-1} \frac{2n-2j -1}{n-j} \prod_{j=k}^{n-1} \frac{n+\alpha-j}{n-j}.
\end{align*}
Since $k< n$ and $\alpha\ge 0$, this inequality clearly holds 
and hence we arrive at the desired inequality, which finishes the proof of \eqref{eq:HeatEst1}.
\end{proof}

Let us also state explicitly the following result.

\begin{corollary}
If  $\alpha\ge 0$, then
\be\label{eq:HeatEst1wt}
\|\E^{-t \wt{H}_\alpha}\|_{\ell^1(\sigma_\alpha^2)\to \ell^\infty} = \frac{1}{(1+t)^{1+\alpha}},\quad t>0.
\ee
If $\alpha\in(-1,0)$, then
\be\label{eq:HeatEst2wt}
\|\E^{-t H_\alpha}\|_{\ell^1(\sigma_\alpha^2)\to \ell^\infty} \ge \frac{1}{(1+t)^{1+\alpha}},\quad t>0.
\ee
\end{corollary}

\subsection{Transience and stochastic completeness}

The operator $\wt{H}_\alpha$ generates a random walk on $\Z_{\ge 0}$ (see Remark \ref{rem:RW}). Explicit form of the heat kernel enables us to characterize basic properties of the corresponding random walk.

\begin{lemma}\label{lem:RW}
The Markovian semigroup $(\E^{-t\wt{H}_\alpha})_{t>0}$ 
 is transient if and only if $\alpha>0$. 
Moreover, it is stochastically complete (conservative) for all $\alpha>-1$.
\end{lemma}

\begin{proof}
The first claim can be deduced either from Lemma \ref{lem:Greenat0} or Lemma \ref{lem:diagest01} (see \cite[Lemma 1.5.1]{fuk10}). 

Recall that $(\E^{-t\wt{H}_\alpha})_{t>0}$ is called stochastically complete (conservative) if 
\be\label{eq:SCdef}
\E^{-t\wt{H}_\alpha} \id = \id
\ee
for some (and hence for all) $t>0$. Taking into account Remark \ref{rem:heat=Meixner}, 
\eqref{eq:SCdef} follows from \dlmf{18.23.3}. Indeed, the generating function for Meixner polynomials is given by 
\begin{align*}
\sum_{n\ge 0}\frac{(\beta)_n}{n!} M_n(x;\beta,c) z^n = \Big(1-\frac{z}{c}\big)^x (1-z)^{-x-\beta},
\end{align*}
where $x\in\Z_{\ge 0}$ and $|z|<1$. However, by \eqref{eq:heatWTaction} and \eqref{eq:heat=Meixner}, we get
\begin{align*}
(\E^{-t\wt{H}_\alpha}\id)_n & = \sum_{m\ge 0} \E^{-t\wt{H}_\alpha}(n,m) \sigma_\alpha(m)^2  \\
& = \frac{1}{(1+t)^{\alpha+1}}\Big(\frac{t}{t+1}\Big)^{n} \sum_{m\ge 0} \frac{(\alpha+1)_m}{m!} M_n\Big(m;1+\alpha,\frac{t^2}{t^2-1}\Big) \Big(\frac{t}{t+1}\Big)^{m} \\
& = \frac{1}{(1+t)^{\alpha+1}}\Big(\frac{t}{t+1}\Big)^{n} \sum_{m\ge 0} \frac{(\alpha+1)_m}{m!} M_m\Big(n;1+\alpha,\frac{t^2}{t^2-1}\Big) \Big(\frac{t}{t+1}\Big)^{m} \\
& = \frac{1}{(1+t)^{\alpha+1}}\Big(\frac{t}{t+1}\Big)^{n} \Big(1-\frac{t-1}{t}\Big)^n \Big(1-\frac{t}{t+1}\Big)^{-n-\alpha-1} = 1. \qedhere
\end{align*}
\end{proof}

\begin{remark}\label{rem:SC}
A few remarks are in order.
\begin{itemize}
\item[(i)] For large classes of graphs there are rather transparent geometric criteria for stochastic completeness (e.g., via volume growth). In particular, applying \cite[Theorem 5]{klw} (see also \cite[Chapter 9]{klwBook}) to \eqref{eq:StringDiff}, we get 
\begin{align*}
\sum_{n\ge 0}\sigma_\alpha(n)^2 \omega_\alpha(n) = \sum_{n\ge 0}\frac{1}{n+\alpha+1} = \infty,
\end{align*}  
which implies stochastic completeness. Taking into account \eqref{eq:SCdef} this provides another derivation of the generating function for Meixner polynomials.
\item[(ii)] Notice that according to the Khas'minskii-type theorem, stochastic completeness implies uniqueness of the Cauchy problem for \eqref{eq:heatWT} in $\ell^\infty$ (respectively, for \eqref{eq:heat} in $\ell^\infty(\sigma_\alpha)$). 
 We do not plan to discuss this issue here and only refer for further details to, e.g., \cite{fuk10}, \cite{klw}, \cite[Chapter 7]{klwBook}.
\end{itemize}
\end{remark}

\section{Eigenvalue estimates}\label{sec:V}

Consider the perturbed operator
\be\label{eq:HwithV}
H_{\alpha,V}:=H_\alpha - V,
\ee
where $V$ is a multiplication operator on $\ell^2(\Z)$ given by
\be\label{eq:Vmult}
(Vu)_n := v_n u_n,\quad n\in \Z_{\ge0}.
\ee
We shall always assume that $(v_n)_{n\ge 0}$ is a real sequence. If $V$ is unbounded (i.e., $(v_n)_{n\ge 0}$ is unbounded),  we define the operator $H_{\alpha,V}$ as the maximal operator (analogous to $H_{\alpha}$)  and this operator is self-adjoint according to the Carleman test (see, e.g., \cite{akh}). We shall denote the total multiplicity of the negative spectrum of $H_{\alpha,V}$ by $\kappa_-(H_{\alpha,V})$. Notice that the spectrum of a semi-infinite Jacobi matrix is always simple and hence $\kappa_-(H_{\alpha,V})$
is the number of negative eigenvalues of $H_{\alpha,V}$ if the negative spectrum of $H_{\alpha,V}$ is discrete. 

For a real sequence $v = (v_n)_{n\ge 0}\subset \R$, denote $v^{\pm}:= (|v|\pm v)/2$. Let also $V^\pm$ be the corresponding multiplication operators.  Since $V=V^+ - V^-$, the min-max principle implies the following standard estimate
\be\label{eq:EVest+-}
\kappa_-(H_{\alpha,V}) \le \kappa_-(H_{\alpha,V^+}).
\ee

\subsection{Rank one perturbations}
We begin with the simplest possible case, which however demonstrates several important features.
Let us consider the operator 
\be\label{eq:Hrank1}
H_\alpha(v_n) : = H_\alpha - v_n\langle \cdot, \delta_n\rangle \delta_n,\quad v_n>0.
\ee
Thus, $H_\alpha(v_n)$ is a rank one perturbation of the operator $H_\alpha$ (the corresponding matrix coincides with \eqref{eq:H0} except the coefficient $h^{(\alpha)}_{n,n}$ replaced by $h^{(\alpha)}_{n,n} - v_n$).

\begin{lemma}\label{lem:rank1est}
Let $v_n>0$. If $\alpha\in (-1,0]$, then 
\be
\kappa_-(H_\alpha(v_n)) = 1.
\ee
If $\alpha>0$, then 
\be
\kappa_-(H_\alpha(v_n)) = \begin{cases} 0, & v_n\in (0,\alpha], \\ 1, & v_n>\alpha .\end{cases}
\ee
\end{lemma}

\begin{proof}
Since $H_\alpha$ is a positive operator, $\kappa_-(H_\alpha(v_n)) \le 1$. Suppose that $E<0$ is an eigenvalue of $H_\alpha(v_n)$, that is, there exists $f\in \ell^2(\Z_{\ge 0})$ such that 
\begin{align*}
H_\alpha(v_n) f = Ef.
\end{align*}
Therefore, we get
\begin{align*}
(H_\alpha - E)f  = v_n\langle f, \delta_n\rangle \delta_n , 
\end{align*}
which shows that 
\begin{align*}
f = v_n\langle f, \delta_n\rangle (H_\alpha - E)^{-1} \delta_n.
\end{align*}
Hence by \eqref{eq:H0resolv}
\be\label{eq:VviaG}
\frac{1}{v_n} = \big\langle (H_\alpha - E)^{-1} \delta_n ,\delta_n \big\rangle = G_\alpha(E;n,n).
\ee
Since $H_\alpha$ is positive, $G_\alpha(\cdot;n,n)$ is increasing on $(-\infty,0)$. Moreover, $G_\alpha(E;n,n) \to 0$ as $E\to -\infty$. 
Therefore, by \eqref{lem:Greenat0} $G_\alpha(\cdot;n,n)$ maps $(-\infty,0)$ onto $\R_{>0}$ if $\alpha\in (-1,0]$ and onto $(0,\alpha^{-1})$ if $\alpha>0$, which implies that \eqref{eq:VviaG} has a solution exactly when 
\begin{align*}
v_n \in \begin{cases} (0,\infty), & \alpha\in (0,1], \\ (\alpha,\infty), & \alpha>0 .\end{cases}
\end{align*}
This immediately proves the desired claim.
\end{proof}

\begin{remark}
Notice that in the case $n=0$ the corresponding eigenvalue $\lambda(v_0)$ is explicitly given by
\be
\lambda(v_0) = m^{-1}_\alpha(1/v_0) < 0,
\ee
where $m_\alpha$ is the Weyl function \eqref{eq:malpha}. The case $\alpha = 0$ has been addressed in \cite{ks15b}.
\end{remark}

\subsection{CLR and Bargmann-type bounds}

We begin with the following extension of Lemma \ref{lem:rank1est}.

\begin{lemma}\label{lem:criticalA}
Let $\alpha\in (-1,0]$. If $v=v^+\not\equiv 0$, then $\kappa_-(H_{\alpha,V})\ge 1$. 
\end{lemma}

\begin{proof}
The proof is immediate from Lemma \ref{lem:rank1est} and the min-max principle.
\end{proof}

Our main aim is to extend the second claim in Lemma \ref{lem:rank1est} to more general potentials. We begin with the following result, which may be seen as the analog of the famous Cwikel--Lieb--Rozenblum bound (see \cite{fls}, \cite{ls97}, \cite[Theorem XIII.12]{rsIV}).

\begin{theorem}\label{th:CLR}
Let $\alpha>0$ and $v_+\in \ell^{1+\alpha}(\Z_{\ge 0};\sigma_\alpha^2)$. Then the operator $H_{\alpha,V}$ is bounded from below, its negative spectrum is discrete and, moreover, there is a constant $C=C(\alpha)>0$ (independent of $v$) such that 
\be\label{eq:CLR}
\kappa_-(H_{\alpha,V}) \le C(\alpha) \sum_{n\ge 0} (v_n^+)^{1+\alpha}\frac{(\alpha+1)_n}{n!}.
\ee
\end{theorem} 

\begin{proof}
By \eqref{eq:EVest+-}, we can assume that $V= V^+$, that is, $v_n=v_n^+ \ge 0$ for all $n\ge 0$. 
Taking into account that the operator $\wt{H}_\alpha$ is Markovian, by Varopoulos's theorem (see \cite[Theorem II.5.2]{var}), \eqref{eq:HeatEst1wt} is equivalent to the validity of the following Sobolev-type inequality
\be\label{eq:Sobalpha}
\Big( \sum_{n\ge 0} |u_n|^{2+\frac{2}{\alpha}}\frac{(\alpha+1)_{n}}{n!}\Big)^{\frac{2\alpha}{2\alpha+2}} \le C_1(\alpha) \sum_{n\ge 0}\frac{(\alpha+1)_{n+1}}{n!}|u_n - u_{n+1}|^2,
\ee
for all $u \in \dom(\wt{\gt}_\alpha)$ and the constant $C_1(\alpha)$ depends only on $\alpha$. By Theorem 1.2 from \cite{ls97}, the latter implies that for each $v$ with $v_+\in \ell^{1+\alpha}(\Z_{\ge 0};\sigma_\alpha^2)$ 
the operator $\wt{H}_{\alpha,V} = \wt{H}_\alpha - V$ is bounded from below, its negative spectrum is discrete and, moreover, there is a constant $C=C(\alpha)>0$ such that 
\be\label{eq:CLRtild}
\kappa_-(\wt{H}_{\alpha,V}) \le C(\alpha) \sum_{n\ge 0} v_n^{1+\alpha}\frac{(\alpha+1)_n}{n!}.
\ee
It remains to notice that the operators $\wt{H}_{\alpha,V}$ and $H_{\alpha,V}$ are unitarily equivalent since $V$ commutes with $\cL_\alpha$ and thus 
\begin{equation*}
\kappa_-(H_{\alpha,V}) = \kappa_-(H_\alpha - V) = \kappa_-(\cL_\alpha (\wt{H}_\alpha - V) \cL_\alpha^{-1}) = \kappa_-(\wt{H}_\alpha - V)
= \kappa_-(\wt{H}_{\alpha,V}).\qedhere
\end{equation*}
\end{proof}

\begin{remark}\label{rem:CLRconst}
Taking into account \eqref{eq:HeatEst1wt}, $C_1(\alpha) \ge \frac{1}{1+\alpha}$ (see \cite[Remark 2.1]{fls}). Moreover, the constants $C(\alpha)$ and $C_1(\alpha)$ satisfy (see \cite[Theorem 2.1]{fls})
\be
C_1(\alpha)^{1+\alpha} \le C(\alpha) \le \E^\alpha C_1(\alpha)^{1+\alpha}.
\ee
The optimal constants in \eqref{eq:CLR} and \eqref{eq:Sobalpha} remain an open problem. 
\end{remark}

We finish this section with another estimate, which can be seen as the analog of the Bargmann bound (see, e.g., \cite[Theorem XIII.9(a)]{rsIV}.

\begin{theorem}\label{th:kappa-}
If $\alpha>0$, then
\be\label{eq:bargm}
\kappa_-(H_{\alpha,V}) \le \frac{1}{\alpha}\sum_{n\ge 0} v_n^+\, \frac{(\alpha+1)_n}{n!}. 
\ee
\end{theorem}

\begin{proof}
Again, by \eqref{eq:EVest+-} it suffices to prove \eqref{eq:bargm} for $V=V^+$. 
Let $\varepsilon>0$. It is the standard Birman--Schwinger argument (see \cite{simon}) that $-E<0$ is an eigenvalue of $H_{\alpha, \varepsilon V}$ if and only if  $\varepsilon^{-1}$ is the eigenvalue of $V^{1/2} (H_{\alpha} + E)^{-1} V^{1/2}$. Therefore, if the operator $V^{1/2} H_{\alpha}^{-1} V^{1/2}$ extends to a bounded operator on $\ell^2(\Z_{\ge 0})$ (notice that $H_{\alpha}^{-1}$ is densely defined on $\ell^2(\Z_{\ge 0})$ since $0$ is not an eigenvalue of $H_\alpha$), then the number of negative eigenvalues of $H_{\alpha,V}$ equals the number of eigenvalues of  $V^{1/2} H_{\alpha}^{-1} V^{1/2}$ which are greater than $1$. And hence
\begin{align*}
\kappa_-(H_{\alpha,V}) \le {\rm tr}\, (V^{1/2} H_{\alpha}^{-1} V^{1/2}).
\end{align*}
Taking into account \eqref{eq:Hfactor}, we (at least formally) get (however, see Remark \ref{rem:KKcomp} below) 
\begin{align*}
V^{1/2} H_\alpha^{-1} V^{1/2} = V^{1/2} \cL_\alpha (I-\cS)^{-1} \cW_\alpha (I-\cS^\ast)^{-1} \cL_\alpha V^{1/2},
\end{align*}
and hence the trace of this operator is given explicitly by
\begin{align*}
{\rm tr}(V^{1/2} H_\alpha^{-1} V^{1/2})  & = \sum_{n\ge 0}\langle V^{1/2} H_\alpha^{-1} V^{1/2} \delta_n,\delta_n\rangle_{\ell^2}\\
& = \sum_{n\ge 0} \|\cW_\alpha^{1/2}(I-\cS^\ast)^{-1}\cL_\alpha V^{1/2} \delta_n\|_{\ell^2}^2\\
& = \sum_{n\ge 0} v_n\,\frac{(\alpha+1)_n}{n!} \|\cW_\alpha^{1/2}(I-\cS^\ast)^{-1} \delta_n\|_{\ell^2}^2\\
&= \sum_{n\ge 0} v_n\,\frac{(\alpha+1)_n}{n!}  \sum_{k=0}^{n} \frac{k!}{(\alpha+1)_{k+1}}.
\end{align*}
Thus we get 
\be\label{eq:kappa-}
\kappa_-(H_{\alpha,V}) \le \sum_{n\ge 0} v_n\, \frac{(\alpha+1)_n}{n!}  \sum_{k=0}^{n} \frac{k!}{(\alpha+1)_{k+1}}.
\ee
Now take into account that by \dlmf{15.4.20}
\begin{align*}
\sum_{k\ge 0} \frac{k!}{(\alpha+1)_{k+1}} = \frac{1}{\alpha+1} \sum_{k\ge 0} \frac{k!}{(\alpha+2)_{k}} 
= \frac{1}{\alpha+1}\hyp21{1,1}{\alpha+2}{1} 
= \frac{1}{\alpha}.
\end{align*}
Combining the latter with \eqref{eq:kappa-}, we arrive at the desired estimate.
\end{proof}

\begin{remark}\label{rem:KKcomp}
A few remarks are in order.
\begin{itemize}
\item[(i)] 
Using the string factorization \eqref{eq:Hfactor}, it follows from \cite[\S 13]{kakr74} that the operator $V^{1/2} H_\alpha^{-1} V^{1/2} $ with $\alpha>0$ is compact exactly when 
\begin{align*}
\sum_{k=0}^n v_k l_\alpha(k) \sum_{k >n} \omega_\alpha(k) 
   = \sum_{k=0}^n v_k \frac{(\alpha+1)_k}{k!} \sum_{k >n} \frac{k!}{(\alpha+1)_{k+1}} = o(1)
\end{align*}
as $n\to \infty$. In particular, the inclusion $v\in \ell^1(\Z_{\ge 0})$ would imply compactness of  $V^{1/2} H_{\alpha}^{-1} V^{1/2}$ if $\alpha>0$.
\item[(ii)]
 Clearly, $\ell^1(\sigma_\alpha^2)$ is contained in  $\ell^{1+\alpha}(\sigma_\alpha^2)$ if $\alpha>0$ and hence \eqref{eq:CLR} applies to a wider class of potentials than \eqref{eq:bargm}. However, this embedding is not continuous. Moreover, we do not know the optimal $C(\alpha)$ in \eqref{eq:CLR}.
\item[(iii)] 
The case $\alpha\in (-1,0]$ requires different considerations and it will be considered elsewhere. However, let us mention that using the Birman--Schwinger principle and applying the commutation to the string factorization, one can show that for $\alpha\in (-1,0]$,
\be\label{eq:bargm2}
\kappa_-(H_{\alpha,V}) \le 1+ \sum_{n\ge 0} \frac{n!}{(\alpha+1)_{n+1}} \sum_{k\ge n} v_n^+\, \frac{(\alpha+1)_k}{k!},
\ee
where the second summand on the RHS in \eqref{eq:bargm2} is dual to the RHS \eqref{eq:kappa-}.
\item[(iv)] 
The study of spectral types (ac-spectrum, sc-spectrum etc.) of a positive spectrum of $H_{\alpha,V}$ is beyond the scope of the present paper, however, see the recent preprint \cite{yaf}. 
\end{itemize}
\end{remark}

\subsection{Hardy inequality}\label{ss:Hardy}

Our final goal is to provide the optimal Hardy inequality for the operator $H_\alpha$ (the optimality is understood in the sense of \cite[\S 1.2]{kpp18}). For each $\alpha>0$, define the following weights $v_\alpha = (v_\alpha(n))_{n\ge 0}$,
\begin{align}\label{eq:vHardy}
v_\alpha(n) = \frac{q_\alpha(n) + q_\alpha(n+1)}{2},\qquad n\ge 0,
\end{align}
where $q_\alpha = (q_\alpha(n))_{n\ge 0}$ is a positive sequence given by
\begin{align}\label{eq:Qun}
q_\alpha(n) 
= 2n+ \alpha - \sqrt{(2n+ \alpha)^2 - \alpha^2} = \frac{\alpha^2}{2n+ \alpha + \sqrt{(2n+ \alpha)^2 - \alpha^2}}.
\end{align}

\begin{theorem}\label{th:Hardy}
Let $\alpha>0$. The weight $v_\alpha$ is the optimal Hardy weight for $H_\alpha$, that is, the following assertions hold true:
\begin{itemize}
\item[(i)] The operator $H_\alpha - V_\alpha$ is nonnegative and the Hardy-type inequality 
\begin{align}\label{eq:HardyIn}
 \sum_{n\ge 0} \big| \sqrt{n+\alpha+1} u_n - \sqrt{n+1} u_{n+1}\big|^2 \ge \sum_{n\ge 0} v_\alpha(n) |u_n|^2
\end{align}
holds true for all $u\in \ell^2_c(\Z_{\ge 0})$.
\item[(ii)] For any positive $v = (v_n)_{n\ge 0}$ such that $V\neq V_\alpha$ and $V\ge V_\alpha$ the operator $H_\alpha - V$ in no longer nonnegative in $\ell^2(\Z_{\ge 0})$. 
\item[(iii)] For any $\lambda>0$ and any finite subset $X\subset \Z_{\ge 0}$, $H_\alpha - V_\alpha \ge \lambda V_\alpha$ fails to hold on $\ell^2_c(\Z_{\ge 0}\setminus X)$ (the inequality is understood in the form sense).
\end{itemize}
\end{theorem}

\begin{remark}\label{rem:HardyAss}
Before proving the above result let us briefly comment on the asymptotic behavior of $v_\alpha$ for large $n$.
 Clearly, 
\begin{align}\label{eq:HardyAssMain}
v_\alpha(n) = \frac{\alpha^2}{4n} + o(n^{-1})
\end{align}
as $n\to \infty$. On the other hand, using the Taylor series  expansion
\begin{align*} 
x - \sqrt{x^2 - a^2} = x\big( 1- \sqrt{1 - (a/x)^2}\big) =  \frac{x}{2}\sum_{k\ge 1}\frac{(1/2)_{k-1}}{k!}\Big( \frac{a}{x}\Big)^{2k},\quad \Big|\frac{a}{x}\Big|<1,
\end{align*}
we immediately get 
\begin{align}\label{eq:Hardy01C}
v_\alpha(n) = \frac{1}{4}\sum_{k\ge 1}\frac{(1/2)_{k-1}\alpha^{2k}}{k!} \left(\frac{1}{(2n+ \alpha)^{2k-1}} + \frac{1}{(2n+ \alpha + 2)^{2k-1}}\right),
\end{align}
for all $n\ge 0$. Combining \eqref{eq:Hardy01C} with \eqref{eq:HardyAssMain}, we arrive at the estimate
\begin{align}
0< v_\alpha(n) - \frac{\alpha^2}{4}\left(\frac{1}{2n+ \alpha} + \frac{1}{2n+ \alpha + 2}\right) = \OO(n^{-3}), 
\end{align}
as $n\to \infty$. The latter together with \eqref{eq:HardyIn} implies the following Hardy inequality
\begin{align}\label{eq:HardyIn02}
 \sum_{n\ge 0} \big| \sqrt{n+\alpha+1} u_n - \sqrt{n+1} u_{n+1}\big|^2 > \frac{\alpha^2}{2}\sum_{n\ge 0} \frac{|u_n|^2}{2n+\alpha+1} ,
\end{align}
which holds true for all $0\neq u\in \ell^2_c(\Z_{\ge 0})$.
\end{remark}
 
\begin{proof}[Proof of Theorem \ref{th:Hardy}]
Taking into account that \eqref{eq:HardyIn} is noting but
\[
\gt_\alpha[u] \ge \langle V_\alpha u,u\rangle_{\ell^2},
\]
then using the operator $\cL_\alpha$ given by \eqref{eq:La} as well as the equality \eqref{eq:wtHalpha}, the above inequality is equivalent to 
\begin{align*}
\wt{\gt}_\alpha[u] \ge \langle V_\alpha u,u\rangle_{\ell^2(\sigma_\alpha^2)}.
\end{align*}
By \eqref{eq:wtform}, the latter reads 
\begin{align}\label{eq:HardyInB}
\sum_{n\ge 0} \frac{(\alpha+1)_{n+1}}{n!}|u_n-u_{n+1}|^2 \ge \sum_{n\ge 0}\frac{(\alpha+1)_{n}}{n!}v_\alpha(n)|u_n|^2.
\end{align}
The proof of this inequality is based on \cite[Theorem~1.1]{kpp18}. Consider the sequences $f=(f_n)_{n\ge 0}$ and $g=(g_n)_{n\ge 0}$ given by
\begin{align*}
f_n & = 1, & g_n = \frac{1}{\sigma_\alpha(n)^2} = \frac{n!}{(\alpha+1)_{n}}, 
\end{align*}
for all $n\ge 0$. It is straightforward to verify that $(\wt{\tau}_\alpha f)_n = 0$ for all $n\ge 0$ and $(\wt{\tau}_\alpha g)_n = 0$ for all $n\ge 1$, where $\wt{\tau}_\alpha$ is given by \eqref{eq:StringDiff} (on the other hand, notice that $f = \cL_\alpha^{-1} P_\alpha(0)$ and $g = \alpha\cL_\alpha^{-1}\Psi_\alpha(0)$, where $P_\alpha(0) = (P_{\alpha,n}(0))_{n\ge 0}$ and $\Psi_\alpha(0) = (\Psi_{\alpha,n}(0))_{n\ge 0} = (Q_{\alpha,n}(0) + m_\alpha(0) P_{\alpha,n}(0))_{n\ge 0}$). Taking into account that $g_n$ is strictly decreasing as $n\to \infty$ (since $\alpha>0$), $f$ and $g$ satisfy the assumptions of \cite[Theorem~1.1]{kpp18} and hence the weight $\wt{V}_\alpha = (\wt{v}_\alpha(n))_{n\ge 0}$ given by 
\begin{align*}
\wt{v}_\alpha(n) = \frac{1}{\sqrt{g_n}}\sum_{k\ge 0\colon |k-n|=1} \frac{\sqrt{g_n} - \sqrt{g_k}}{\omega_\alpha(\min(n,k))},\qquad n\ge 0,
\end{align*}
is the optimal Hardy weight for $\wt{H}_\alpha$ (in the sense of \cite{kpp18}). This in particular implies the validity of the inequality
\begin{align*}
\wt{\gt}_\alpha[u] \ge \sum_{n\ge 0}\wt{v}_\alpha(n)|u_n|^2,
\end{align*}
for all $u\in\ell^2_c(\Z_{\ge 0})$. It remains to notice that
\begin{align*}
\wt{v}_\alpha(0) &= \alpha+1 - \sqrt{\alpha+1} = {v}_\alpha(0), \\
\wt{v}_\alpha(n) &= \frac{(\alpha+1)_n}{(n-1)!}\Big(1 - \sqrt{\frac{n+\alpha}{n}} \Big) + \frac{(\alpha+1)_{n+1}}{n!}\Big(1 - \sqrt{\frac{n+1}{n+\alpha+1}} \Big) 
= \sigma_\alpha(n)^2{v}_\alpha(n),
\end{align*}
for all $ n\ge 1$. This immediately implies that \eqref{eq:HardyInB} is the optimal Hardy inequality for $\wt{H}_\alpha$, which completes the proof.
\end{proof}

\begin{remark}
A few concluding remarks are in order.
\begin{itemize}
\item[(i)] 
The classical Hardy inequality (after a simple change of variables $u_n\mapsto nu_n$) states that the inequality\footnote{It was proved in \cite[Theorem~7.3]{kpp18} that replacing $1/4$ on the RHS \eqref{eq:HardyClass} by the weight $w=(w_n)_{n\ge 0}$, $w_n = n^2(2-\sqrt{1+1/n} - \sqrt{1-1/n})$ is the optimal Hardy inequality.} 
\begin{align}\label{eq:HardyClass}
\sum_{n\ge 0}|(n+1)u_{n+1} - nu_n |^2 \ge \frac{1}{4}\sum_{n\ge 1} |u_n|^2
\end{align}
holds true for all $u\in \ell^2(\Z_{\ge 0})$. Setting $\alpha=1$ in \eqref{eq:HardyIn02} and changing variables $u_n\mapsto \frac{1}{\sqrt{n+1}}u_n$ (cf.  \eqref{eq:HardyInB} with $\alpha=1$), we get the inequality 
\begin{align}\label{eq:HardyIn02a=1}
 \sum_{n\ge 0} (n+1)(n+2)\big|  u_{n+1} - u_n \big|^2 \ge \frac{1}{4}\sum_{n\ge 0}|u_n|^2,
\end{align}
which looks in a certain sense similar to the classical one.
\item[(ii)]
CLR and Bargmann-type bounds can be seen as ``integral" conditions which guarantee the positivity of the perturbed operator $H_{\alpha,V}$ (if the RHS in \eqref{eq:CLR} or \eqref{eq:bargm} is less than $1$, then clearly the corresponding LHS is zero). Hardy-type inequalities allow to obtain ``pointwise" positivity conditions. Namely, 
applying the standard minmax principle, \eqref{eq:HardyIn} implies that
\begin{align}\label{eq:kappa=0}
\kappa_-(H_{\alpha,V}) = 0
\end{align}
whenever $v_n^+ \le v_\alpha(n)$ for all $n\ge 0$. In particular, using \eqref{eq:HardyIn02} we conclude that \eqref{eq:kappa=0} holds true whenever
\[
v_n^+ \le \frac{\alpha^2}{2(2n+\alpha+1)},\qquad \text{for all}\ n\ge 0.
\]
\item[(iii)]
The Hardy inequality implies the following {\em Kneser-type result}:  if 
\begin{align}
\limsup_{n\ge 0} nv_n^+ < \frac{\alpha^2}{4},
\end{align}
then $\kappa_-(H_{\alpha,V}) <\infty$. 
Conversely, if there is $\varepsilon>0$ such that $v_n \ge \frac{\alpha^2+\varepsilon}{n}$ for all large enough $n$, then $\kappa_-(H_{\alpha,V}) = \infty$.
\item[(iv)]
If $\alpha\in (-1,0]$, then Lemma \ref{lem:criticalA} implies that for each $v=v_+\not\equiv 0$ there is $u\in \ell^2_c(\Z_{\ge 0})$ such that 
\[
 \sum_{n\ge 0} \big| \sqrt{n+\alpha+1} u_n - \sqrt{n+1} u_{n+1}\big|^2 < \sum_{n\ge 1} v_n |u_n|^2.
\]
In the terminology of \cite{kpp18,kpp20} the latter means that $H_\alpha$ is {\em critical} for all $\alpha\in (-1,0]$. However, choosing $g = g_\alpha = (g_\alpha(n))_{n\ge 0}$ as
\begin{align*}
g_\alpha(0) & = 0, & g_\alpha(n) & = \sum_{k=1}^n \omega_\alpha(k),\ n\ge 1,
\end{align*}
it is straightforward to check that for all $\alpha\in (-1,0]$, $g_\alpha$ satisfies the assumptions of Theorem~1.1 from \cite{kpp18}. Therefore, by  \cite[Theorem~1.1]{kpp18}, the weight 
\begin{align}\label{eq:HardyWTA}
\wt{v}_\alpha(n) = \frac{1}{\sqrt{g_\alpha(n)}}\sum_{k\ge 0\colon |k-n|=1} \frac{\sqrt{g_\alpha(n)} - \sqrt{g_\alpha(k)}}{\omega_\alpha(\min(n,k))},\qquad n\ge 1,
\end{align}
is the optimal Hardy weight and the following  optimal Hardy inequality (cf. \eqref{eq:HardyInB}) holds true
\begin{align}\label{eq:HardyInA}
 \sum_{n\ge 0} \frac{(\alpha+1)_{n+1}}{n!}|u_n-u_{n+1}|^2 \ge \sum_{n\ge 1} \wt{v}_\alpha(n) |u_n|^2,
\end{align}
however,  for all $u\in \ell^2_c(\Z_{\ge 0})$ with $u_0=0$.  
In particular, for $\alpha=0$, $g_0(n) = \sum_{k=1}^n \frac{1}{k} = {\rm h}_n$ are the {\em harmonic numbers} and the corresponding inequality is 
\[
\sum_{n\ge 0}(n+1)|u_n-u_{n+1}|^2 \ge \sum_{n\ge 1} v_0(n) |u_n|^2,\quad u_0=0,
\]
where 
\begin{align*}
v_0(n) & = q_0(n) + q_0(n+1), & q_0(n) & = n - \sqrt{n^2 - n{\rm h}_n} ,\ \ n\ge 1.
\end{align*}
\end{itemize}
\end{remark}

\appendix

\section{Jacobi polynomials}\label{app:jacobi}
For $\alpha$, $\beta>-1$, let
$w^{(\alpha,\beta)}(x)= (1-x)^\alpha (1+x)^\beta$ for $x\in (-1,1)$
be a Jacobi weight.
The corresponding orthogonal polynomials $P_n^{(\alpha,\beta)}$,
normalized by
\be\label{eq:normal}
P_n^{(\alpha,\beta)} (1) = \binom{n+\alpha}{n} = \frac{(\alpha+1)_n}{n!}
\ee
for all $n\ge0$ (see \eqref{K6} for notation of Pochhammer symbols
and binomial coefficients), are called the {\em Jacobi polynomials}.
They are expressed as (terminating) Gauss hypergeometric
series \eqref{K7}
by \cite[(4.21.2)]{sz}
\be\label{eq:01}
\frac{P_n^{(\alpha,\beta)}(x)}{P_n^{(\alpha,\beta)}(1)} =
\hyp21{-n,n+\alpha+\beta+1}{\alpha+1}{\frac{1-x}{2}}.
\ee
They also satisfy Rodrigues' formula \cite[(4.3.1), (4.3.2)]{sz}
\begin{align}
P_n^{(\alpha,\beta)}(x) &=\sum_{k=0}^n \binom{n+\alpha}{n-k} \binom{n+\beta}{k} \left(\frac{x-1}{2}\right)^k\left(\frac{x+1}{2}\right)^{n-k}\label{K5}\\[1mm]
&= \frac{(-1)^n}{2^n n!}(1-x)^{-\alpha}(1+x)^{-\beta} \frac{d^n}{dx^n} \big\{(1-x)^{\alpha+n}(1+x)^{\beta+n} \big\}.\label{eq:01a}
\end{align}
This formula immediately implies
\be\label{eq:symmetry}
P_n^{(\alpha,\beta)}(-x) = (-1)^nP_n^{(\beta,\alpha)}(x),
\ee
and hence 
\be\label{eq:normal_b}
P_n^{(\alpha,\beta)} (-1) =(-1)^n \binom{n+\beta}{n} =
(-1)^n\,\frac{(\beta+1)_n}{n!}\,.
\ee
Jacobi polynomials include the Chebyshev polynomials, the ultraspherical (Gegenbauer) polynomials,  
and the Legendre polynomials (see \cite{dlmf}, \cite{sz} for further details).

\bigskip
{\bf Acknowledgments.}
I am grateful to Matthias Keller and Noema Nicolussi for numerous useful discussions. I am also indebted to the referee for the careful reading of the manuscript and remarks.


\end{document}